\documentclass{amsart}

\usepackage{amssymb,graphicx, color}
\usepackage{xypic}
\usepackage[all]{xy}
\usepackage{lscape}
\usepackage{longtable}

\newtheorem{thm}{Theorem}[section]

\newtheorem{lem}[thm]{Lemma}
\newtheorem{cor}[thm]{Corollary}

\newtheorem{prop}[thm]{Proposition}

\theoremstyle{definition}
\newtheorem{defn}[thm]{Definition}

\newtheorem{rem}[thm]{Remark}          

\newtheorem*{ack}{Acknowledgments}      

\newtheorem{defn-thm}[thm]{Definition--Theorem}  
\newtheorem{defn-lem}[thm]{Definition--Lemma}  

\theoremstyle{remark}
\newtheorem{claim}[thm]{Claim}

\renewcommand{\o}[0]{{\mathcal O}} 

\renewcommand{\L}[0]{{\mathcal L}} 
 
\newcommand{\I}[0]{{\mathcal I}}

\newcommand{\p}[0]{{\mathbb P}}
\newcommand{\f}[0]{{\varphi}}

\newcommand{\T}[0]{{\mathbb T}}
\newcommand{\map}[0]{\dasharrow}

\newcommand{\C}{\mathbb{C}}
\def\rig#1{\smash{ \mathop{\longrightarrow}\limits^{#1}}}

\newcommand{\sL}{\mathcal{L}}

\newcommand{\sH}{\mathcal{H}}

\def\loccoh#1.#2.#3.#4.{H^{#1}_{#2}(#3,#4)}

\DeclareMathAlphabet{\mathchanc}{OT1}{pzc}%
                                {m}{it}

\newcommand{\Sec}[0]{\operatorname{{\mathbb S}ec}}
\newcommand{\Bs}[0]{\operatorname{{Bs}}}
\newcommand{\Span}[1]{{\langle#1\rangle}}
\newcommand{\Gr}[0]{{{\mathbb G}r}}
\numberwithin{equation}{section}
\begin{document}
\bibliographystyle{amsalpha}

\title[Waring decompositions for a polynomial vector]{On the number of Waring decompositions for a generic polynomial vector\\
}

\author[E. Angelini, F. Galuppi, M. Mella]{Elena Angelini, Francesco Galuppi, Massimiliano Mella}
\address[E.Angelini, F. Galuppi, M. Mella]{ Dipartimento di  Matematica e Informatica\\ Universit\`a di Ferrara\\ Via Machiavelli 35\\ 44121 Ferrara, Italia}
\email{elena.angelini@unife.it, francesco.galuppi@unife.it, mll@unife.it}

\author[G. Ottaviani]{Giorgio Ottaviani}
\address[G. Ottaviani]{Dipartimento di Matematica e Informatica 'Ulisse Dini'\\ 
Universit\`{a} di Firenze \\ Viale Morgagni 67/A \\  Firenze , Italia}
\email{ottavian@math.unifi.it}

\begin{abstract} We prove that a general polynomial vector $(f_1, f_2, f_3)$ in three homogeneous variables of degrees $(3,3,4)$
has a unique Waring decomposition of rank 7. This is the first new case we are aware, and likely the last one, after five examples known since 19th century and
the binary case. We prove that there are no identifiable cases among pairs $(f_1, f_2)$ in three homogeneous variables
of degree $(a, a+1)$, unless $a=2$, and we give a lower bound on the number of decompositions. The new example was discovered with Numerical Algebraic Geometry, while its proof needs Nonabelian Apolarity.
\end{abstract}
\maketitle

\section{Introduction}\label{sec:intr}

Let $f_1$, $f_2$ be two general quadratic forms in $n+1$ variables over $\C$.
A well known theorem, which goes back to Jacobi and Weierstrass, says that $f_1$, $f_2$ can be simultaneously diagonalized. More precisely there exist linear forms $l_0,\ldots, l_n$ and scalars  $\lambda_0,\ldots, \lambda_n$ such that
\begin{equation}\label{eq:2quadrics}
\left\{\begin{array}{rcl}f_1&=&\sum_{i=0}^nl_i^2\\
\hspace{0.2cm} \\
f_2&=&\sum_{i=0}^n\lambda_il_i^2\end{array}\right.
\end{equation} 
An important feature is that the forms $l_i$ are unique (up to order) and their equivalence class, up to multiple scalars, depend only on the pencil $\left< f_1, f_2\right>$, hence also $\lambda_i$
are uniquely determined after $f_1$, $f_2$ have been chosen in this order.
The canonical form (\ref{eq:2quadrics}) allows to write easily the 
basic invariants of the pencil, like the discriminant which takes the form $\prod_{i<j}(\lambda_i-\lambda_j)^2$. We call  (\ref{eq:2quadrics}) a (simultaneous) Waring decomposition of
the pair $(f_1, f_2)$.  The pencil $(f_1,f_2)$ has a unique Waring decomposition with $n+1$ summands if and only if its discriminant does not vanish. In the tensor terminology, $(f_1, f_2)$ is {\it generically identifiable}.

We generalize now the decomposition (\ref{eq:2quadrics}) to $r$ general forms, even allowing different degrees. For symmetry reasons, it is convenient not to distinguish $f_1$ from the other $f_j$'s, so we will allow scalars $\lambda^j_i$ to the decomposition of each $f_j$, including $f_1$. To be precise, let $ f=(f_1, \ldots, f_r) $ be a vector of general homogeneous forms of degree $ a_1, \ldots, a_r $ in $n+1$ variables over the complex field $ \C $, i.e. $ f_i \in {\mathrm Sym}^{a_i} \C^{n+1} $ for all $ i \in \{1, \ldots, r\} $. Let assume that $ 2 \leq a_{1} \leq \ldots \leq a_{r} $. 

\begin{defn}\label{simdec}
A \emph{Waring decomposition} of  $ f=(f_1, \ldots, f_r) $ is given by
 linear forms $ \ell_1, \ldots, \ell_k \in \mathbb{P}(\C^\vee) $ and scalars $ (\lambda_{1}^{j}, \ldots, \lambda_k^{j}) \in \mathbb{C}^{k}-\{\underline{0}\} $ with $ j \in \{1, \ldots, r\} $ such that
\begin{equation}\label{eq:dec}
f_{j} = \lambda_{1}^{j}\ell_{1}^{a_{j}}+ \ldots +  \lambda_{k}^{j}\ell_{k}^{a_{j}}
\end{equation}
for all $ j \in \{1, \ldots, r\} $
or in vector notation
\begin{equation}\label{eq:simwaring}
f=\sum_{i=1}^k\left(\lambda_{i}^{1}\ell_{i}^{a_{1}},\ldots, \lambda_{i}^{r}\ell_{i}^{a_{r}}\right)
\end{equation}

The geometric argument in \S \ref{subsec:projbundle} shows that every $f$ has a Waring decomposition. We consider two  Waring decompositions of $f$ as in (\ref{eq:simwaring}) being equal if they differ just by the order of the $k$ summands. The {\it  rank} of
$f$ is the minimum number $k$ of summands appearing in (\ref{eq:simwaring}), this definition coincides with the classical one
in the case $r=1$ (the vector $f$ given by a single polynomial).

\end{defn}
Due to the presence of the scalars $\lambda^j_i$, each form $\ell_{i}$ depends essentially only on $n$ conditions. So the decomposition (\ref{eq:dec}) may be thought as a nonlinear system with $\sum_{i=1}^r{{a_i+n}\choose n}$ data (given by $f_{j}$) and $k(r+n)$ unknowns (given by $kr$ scalars $\lambda^j_i$ and $k$ forms $\ell_{i}$). This is a very classical subject, see for example \cite{Re, Lon, Ro, Sco, Te2}, although in most of classical papers the degrees $a_i$ were assumed equal, with the notable exception of \cite{Ro}.
\begin{defn}\label{d:perfectcases}
Let $ a_1,\ldots, a_r, n$ be as above. 

\noindent The space $ {\mathrm Sym}^{a_1} \C^{n+1}\oplus\ldots\oplus {\mathrm Sym}^{a_r} \C^{n+1}$
is called \emph{perfect} if there exists $k$ such that
\begin{equation}\label{eq:perfect}
\sum_{i=1}^r{{a_i+n}\choose n} = k(r+n) 
\end{equation}
i.e. when (\ref{eq:dec}) corresponds to a square polynomial system.  
\end{defn}

The arithmetic condition (\ref{eq:perfect}) means that $\sum_{i=1}^r{{a_i+n}\choose n}$ is divisible by $(r+n)$, in other terms 
the number of summands $k$ in the system (\ref{eq:dec}) is uniquely determined.

The case with two quadratic forms described in (\ref{eq:2quadrics}) corresponds to $r=2$, $a_1=a_2=2$, $k=n+1$ and it is perfect. The perfect cases are important because, by the above dimensional count, we expect finitely many  Waring decompositions for the generic polynomial vector
in a perfect space $ {\mathrm Sym}^{a_1} \C^{n+1}\oplus\ldots\oplus {\mathrm Sym}^{a_r} \C^{n+1}$.

It may happen that general elements in perfect spaces have no decompositions with the expected
number $k$ of summands, the first example, beside the one of plane
conics, was found by Clebsch in the XIXth century
and regards ternary quartics, where $r=1$, $a_1=4$ and $n=2$. Equation (\ref{eq:perfect}) gives $k=5$ but 
in this case the system (\ref{eq:dec}) has no solutions and indeed $6$ summands are needed to find a Waring decomposition of the general ternary quartic.
It is well known that all the perfect cases with $r=1$ when  the system (\ref{eq:dec}) has no solutions have been determined by Alexander and Hirschowitz,
while more cases for $r\ge 2$ have been found in \cite{CaCh}, where a collection of classical and modern interesting examples is listed.

Still, perfectness is a necessary condition to have finitely many  Waring decompositions.
So two natural questions, of increasing difficulty, arise.
\vskip 0.4cm
{\bf Question 1} Are there other perfect cases for $a_1,\ldots, a_r, n$, beyond (\ref{eq:2quadrics}), where a unique  Waring decomposition
 (\ref{eq:simwaring}) exists for generic $f$, namely where we have generic identifiability ?
\vskip 0.4cm
{\bf Question 2} Compute the number of  Waring decompositions (up to order of summands) for a generic $f$ in any perfect case.
\vskip 0.4cm 
The above two questions are probably quite difficult, but we feel it is worth to state them as guiding problems.
These two questions are open even in the case $r=1$ of a single polynomial. In case $r=1$, Question 1 has a conjectural answer due to the third author,
who proved many cases of this Conjecture in \cite{Me, Me1}. The birational technique used in these papers has been generalized to our setting in \S \ref{sec:ternary} of this paper.
Always in case $r=1$, some number of decompositions for small $a_1$ and $n$ have been computed (with high probability) in \cite{HOOS} by homotopy continuation techniques, with the numerical software Bertini \cite{Be}.

In this paper we contribute to the above two questions. Before stating our conclusions, we still need to expose other known results on this topic.

In the case $n=1$ (binary forms) there is a result by Ciliberto and Russo \cite{CR}
which completely answers our Question 1.

\begin{thm}[Ciliberto-Russo]\label{thm:binforms0}
Let $n=1$. In all the perfect cases there is a unique  Waring decomposition
for generic $f\in {\mathrm Sym}^{a_1}\C^2\oplus\ldots\oplus {\mathrm Sym}^{a_r}\C^2$ if and only if
$ a_{1}+1 \geq \frac{\sum_{i=1}^r(a_i+1)}{r+1} $. (Note the fraction $\frac{\sum_{i=1}^r(a_i+1)}{r+1}$ equals the number $k$ of summands).
\end{thm}

We will provide alternative proofs to Theorem \ref{thm:binforms0} by
using Apolarity, see Theorem \ref{thm:nonabelian_applied2}.

As widely expected, for $n>1$ generic identifiability is
quite a rare phenomenon. They have been extensively
investigated in the XIX$^{\rm th}$ century and at the beginning of the
XX$^{\rm th}$ century and the following are the only discovered cases that we are aware:
\begin{equation}\label{eq:classiclist}
\left\{\begin{array}{ll}
(i)   &({\mathrm Sym}^2\C^n)^{\oplus 2}, \textrm{rank\ } n,\textrm{Weierstrass \cite{We}, as in (\ref{eq:2quadrics})},\\
(ii) & {\mathrm Sym}^5\C^3, \textrm{rank }7, \textrm{Hilbert \cite{Hi}, see also
  \cite{Ri} and \cite{Pa}},\\
(iii)&{\mathrm Sym}^3\C^4, \textrm{rank } 5,\textrm{Sylvester Pentahedral Theorem \cite{Sy}},\\
(iv)& ({\mathrm Sym}^2\C^3)^{\oplus 4}, \textrm{rank\ } 4,\\
(v)&{\mathrm Sym}^2\C^3\oplus {\mathrm Sym}^3\C^3,\textrm{rank\ }4,\textrm{Roberts \cite{Ro}.}
\end{array}\right.
\end{equation}

 The interest in Waring decompositions  was revived by
    Mukai's work on 3-folds, \cite{Mu}\cite{Mu1}. Since then many authors
    devoted their energies to understand, interpret and expand the theory.
Cases $(ii)$ and $(iii)$ in (\ref{eq:classiclist}) were explained by Ranestad and Schreyer in
\cite{RS} by using syzygies, see also \cite{MM} for an approach via
projective geometry and \cite{OO} for a vector bundle approach (called in this paper ``Nonabelian Apolarity'', see \S\ref{sec:Nonabelian}). Case $(v)$ was reviewed in \cite{OS} in
the setting of Lueroth quartics. $(iv)$ is a classical and ``easy'' result, there is a unique Waring decomposition of a
general 4-tuple of ternary quadrics. There is
a very nice geometric interpretation for this latter case. Four points in $\p^5$ define a
$\p^3$ that cuts the Veronese surface in 4 points giving the required
unique decomposition. See Remark \ref{rem:d^n} for a generalization to arbitrary $(d,n)$.  

Our main contribution with respect to unique decompositions is the following new case.
\begin{thm}
  \label{thm:main334} A general $f\in {\mathrm Sym}^3\C^3\oplus {\mathrm Sym}^3\C^3\oplus
  {\mathrm Sym}^4\C^3$ has a unique Waring decomposition of rank 7, namely it is identifiable.
\end{thm}

The Theorem will be proved in the general setting of Theorem \ref{thm:nonabelian_applied2}.
Beside the new example found we think it is important to stress the
way it arised.
We adapted the methods in  \cite{HOOS} to our setting, by using the  software Bertini \cite{Be} and also the package {\it Numerical Algebraic Geometry} \cite{KL}
in Macaulay2 \cite{M2}, with the generous help by Jon Hauenstein and Anton Leykin, who assisted us in writing our first scripts.
The computational analysis of perfect cases of forms on $\C^3$ suggested that for ${\mathrm Sym}^3\C^3\oplus {\mathrm Sym}^3\C^3\oplus
  {\mathrm Sym}^4\C^3$ the
 Waring decomposition is unique. Then we proved it via
Nonabelian Apolarity with the choice of a vector bundle. 
Another novelty of this paper is a unified proof
of almost all cases with a unique Waring decomposition via Nonabelian
Apolarity with the choice of a vector bundle $E$, see Theorem \ref{thm:nonabelian_applied2}. Finally we borrowed a construction
from \cite{MM} to prove, see Theorem \ref{thm:unirat}, that whenever we have uniqueness for rank $k$
then the variety parametrizing Waring decompositions of higher rank is
unirational.

Pick $r=2$ and $n=2$, the space ${\mathrm Sym}^a\C^3\oplus{\mathrm Sym}^{a+1}\C^3$ is perfect
if and only if $a=2t$ is even. All the numerical computations we did suggested that identifiability holds only for $a=2$ (by Robert's Theorem, see (\ref{eq:classiclist}) $(v)$). Once again this
pushed us to prove the non-uniqueness for these pencils of plane curves.
Our main contribution to Question 2 regards this case and it is the following.
\begin{thm}
  \label{th:main_identifi_intro} A general $f\in{\mathrm Sym}^{a}\C^3\oplus{\mathrm Sym}^{a+1}\C^3$ is identifiable
if and only if $a=2$, corresponding to (v) in the list (\ref{eq:classiclist}). Moreover
$f$ has finitely many Waring decompositions if and only if $a=2t$ and in this
case the number of decompositions is at least
$$ \frac{(3t-2)(t-1)}2+1. $$
\end{thm}

We know by equation~(\ref{eq:classiclist})(v) that the bound is sharp for $t=1$ and we verified with high probability, using \cite{Be}, that it 
is attained also for $t=2$. On the other hand we do not expect it to be sharp in general. 
Theorem \ref{th:main_identifi_intro} is proved in section \S
\ref{sec:ternary}. The main idea, borrowed from \cite{Me}, is to
bound the number of decompositions with the degree of a tangential
projection, see Theorem \ref{th:birational_tangent_proj}. To bound the
latter we use a degeneration argument, see Lemma
\ref{lem:birational_degeneration}, that reduces the computation needed
to an intersection calculation on the plane.

\begin{ack} We thank all the participants of the seminar about Numerical Algebraic Geometry held among Bologna, Ferrara, Firenze and Siena in 2014-15, for fruitful and stimulating discussions. We benefit in particular speaking with A. Bernardi, C. Bocci, A. Calabri, L. Chiantini. We thank J. Hauenstein and A. Leykin for their help with our first numerical computations. All the authors are members of GNSAGA-INDAM.
\end{ack}

\section{The Secant construction}\label{sec:secant}
\subsection{Secant Varieties}
Let us recall, next, the main definitions and results concerning secant varieties.
Let $\Gr_k=\Gr(k,N)$ be the Grassmannian of $k$-linear spaces in $\p^N$.
Let $X\subset\p^{N}$ be an irreducible variety
$$\Gamma_{k+1}(X)\subset X\times\cdots\times X\times\Gr_k,$$
 the closure of the graph of
$$\alpha:(X\times\cdots\times X)\setminus\Delta\to \Gr_k,$$
taking $(x_0,\ldots,x_{k})$ to the  $[\langle
  x_0,\ldots,x_{k}\rangle]$, for $(k+1)$-tuple of distinct points.
Observe that $\Gamma_{k+1}(X)$ is irreducible of dimension $(k+1)n$. 
Let $\pi_2:\Gamma_{k+1}(X)\to\Gr_k$ be
the natural projection.
Denote by 
$$S_{k+1}(X):=\pi_2(\Gamma_{k+1}(X))\subset\Gr_k.$$
Again $S_{k+1}(X)$ is irreducible of dimension $(k+1)n$.
Finally let 
$$I_{k+1}=\{(x,[\Lambda])| x\in \Lambda\}\subset\p^{N}\times\Gr_k,$$
with natural projections $\pi_i$ onto the factors.
Observe that $\pi_2:I_{k+1}\to\Gr_k$ is a $\p^{k}$-bundle on $\Gr_k$.

\begin{defn}\label{def:secant} Let $X\subset\p^{N}$ be an irreducible variety. The {\it abstract $k$-Secant variety} is
$$\sec_k(X):=\pi_2^{-1}(S_k(X))\subset I_k.$$ While the {\it $k$-Secant variety} is
$$\Sec_k(X):=\pi_1(\sec_k(X))\subset\p^N.$$
It is immediate that $\sec_k(X)$ is a $(kn+k-1)$-dimensional variety with a 
$\p^{k-1}$-bundle structure on $S_k(X)$. One says that $X$ is
$k$-\emph{defective} if $$\dim\Sec_k(X)<\min\{\dim\sec_k(X),N\}$$ and calls $ k $-\emph{defect} the number $$\delta_{k}=\min\{\dim\sec_k(X),N\}-\dim\Sec_k(X).$$
\end{defn}

\begin{rem} Let us stress that in our definition $\Sec_1(X)=X$. A simple but useful feature of the above definition is the following. Let $\Lambda_1$ and $\Lambda_2$ be two distinct
$k$-secant $(k-1)$-linear space to $X\subset\p^{N}$. Let $\lambda_1$ and  $\lambda_2$ be
the corresponding projective $(k-1)$-spaces in $\sec_k(X)$. 
Then we have $\lambda_1\cap \lambda_2=\emptyset$.
\label{re:vuoto} 
\end{rem}
   
Here is the main result we use about secant varieties.
\begin{thm}[Terracini Lemma \cite{Te}\cite{ChCi}] \label{th:terracini}
Let $X\subset\p^{N}$ be an irreducible, projective variety. 
If $p_1,\ldots, p_{k}\in X$ are general points and $z\in\langle
p_1,\ldots, p_{k}\rangle$
 is a general point, then the embedded tangent space at $z$ is
$$\T_z\Sec_k(X) = \langle \T_{p_1}X,\ldots, \T_{p_{k}}X\rangle$$
If $X$ is k-defective, then the general hyperplane $H$ containing
$\T_{z}\Sec(X)$ is tangent to $X$ along a variety
$\Sigma(p_1,\ldots, p_{k})$ of pure, positive dimension, containing
$p_1,\ldots, p_{k}$.
\end{thm}

\subsection{Secants to a projective bundle}\label{subsec:projbundle}
We show a geometric interpretation of the decomposition  (\ref{eq:dec}) by considering the $k$-secant variety to the projective bundle (see \cite[II, \S 7]{Har})
$$ X=\mathbb{P}(\mathcal{O}_{\p^n}(a_{1}) \oplus \ldots \oplus \mathcal{O}_{\p^n}(a_{r}) ) \subset \mathbb{P}\left(H^0\left(\oplus_i \mathcal{O}_{\p^n}(a_{i})\right)\right) = \mathbb{P}^{N-1}, $$
where $N=\sum_{i=1}^r{{a_i+n}\choose n}$. We denote by $\pi\colon X\to\p^n$ the bundle projection. Note that $\dim X=(r+n-1)$ and the immersion in $\mathbb{P}^{N-1}$ corresponds to the canonical invertible sheaf $\o_X(1)$
constructed on $X$ (\cite[II, \S 7]{Har}).

Indeed $X$ is parametrized by
$\left(\lambda^{(1)}\ell^{a_{1}}, \ldots, \lambda^{(r)}\ell^{a_{r}} \right)\in \oplus_{i=1}^rH^0\left(\mathcal{O}_{\p^n}(a_{i})\right)$, where $\ell\in\C^{n+1}$ and $\lambda^{(i)}$ are scalars. $X$ coincides with polynomial vectors of rank $1$, as defined in the Introduction.
It follows that the $k$-secant variety to $X$ is parametrized by
$\displaystyle{\sum_{i=1}^{k}}\left(\lambda_{i}^{1}\ell_{i}^{a_{1}}, \ldots, \lambda_{i}^{r}\ell_{i}^{a_{r}} \right), $ where $\lambda_i^j$ are scalars and $\ell_i\in\C^{n+1}$. In the case $a_i=i$ for $i=1,\ldots, d$, this construction appears already in
\cite{CQU}. Since $X$ is not contained in a hyperplane, it follows that  any
polynomial vector has a Waring decomposition as in (\ref{eq:simwaring}).

 Thus, the number of  decompositions by means of $ k $ linear forms of $ f_{1}, \ldots, f_{r} $ equates the $k$-\emph{secant degree} of $X$. 

If $ a_{i} = a $ for all $ i \in \{1, \ldots, r\} $, then we deal with $ \mathbb{P}^{r-1}\times\p^n $ embedded through the Segre-Veronese map with $ \o(1,a) $, as we can see in Proposition 1.3. of \cite{DF} or in \cite{BBCC}. \\
Moreover, we remark that assuming to be in a perfect case in the sense of Definition \ref{d:perfectcases} is equivalent to the fact that $ \mathbb{P}(\mathcal{O}_{\p^n}(a_{1}) \oplus \ldots \oplus \mathcal{O}_{\p^n}(a_{r}) ) $ is a \emph{perfect} variety, i.e. $(n+r)|N$. \\
{Theorem \ref{thm:binforms0} has the following reformulation
(compare with Claim $5.3.$ and Proposition 1.14 of \cite{CR}) :}

\begin{cor}\label{c:identifiability}
If (\ref{eq:perfect}) and $ a_{1}+1 \geq k $ hold, then $ \mathbb{P}(\mathcal{O}_{\p^1}(a_{1}) \oplus \ldots \oplus \mathcal{O}_{\p^1}(a_{r}) ) $ is $ k $-identifiable, i.e. its $k$-secant degree is equal to $ 1 $.
\end{cor}

\begin{rem}
A formula for the dimension of the $k$-secant variety of the rational normal scroll $X$ for $n=1$ has been given in \cite[pag. 359]{CaJo} (with a sign mistake,
corrected in \cite[Prop. 1.14]{CR}).
\end{rem}

\begin{rem}\label{rem:d^n}
We may consider the Veronese variety $V:=V_{d,n}\subset\p^{{{d+n}\choose{n}}-1}$. Let $s-1={\rm cod} V$ then $s$ general points determine a unique $\p^s$ that intersects $V$ in $d^n$ points. The $d^n$ points are linearly independent
only if $d^n=s$ that is either $n=1$ or $(d,n)=2$. This shows that a
general vector $f=(f_1,\ldots,f_s)$ of forms of degree $d$ admits
${d^n}\choose{s}$ decompositions, see the table at the end of \S\ref{sec:compapp} for some numerical examples. On the other hand, from a
different perspective, dropping the requirement that the linear forms
giving the decompositions are linearly independent, this shows that
there is a unique set of $d^n$ linear forms that decompose the general
vector $f$. Note that this time only the forms and not the coefficient
are uniquely determined. We will not dwell on this point of view here
and left it for a forthcoming paper.
\end{rem}

\section{Nonabelian Apolarity and Identifiability}
\label{sec:binary}\label{sec:Nonabelian}

Let $f\in Sym^dV$. For any $e\in{\mathbb Z}$, Sylvester constructed the catalecticant map
$C_f\colon {\mathrm Sym}^eV^*\to {\mathrm Sym}^{d-e}V$ which is the contraction by $f$. Its main property is the inequality $\mathrm{rk\ }C_f\le\mathrm{rk\ } f$,
where the rank on left-hand side is the rank of a linear map, while the rank on the right-hand side has been defined in the Introduction.
In particular the $(k+1)$-minors of $C_f$ vanish on
the variety of polynomials with rank bounded by $k$, which is $\Sec_k(V_{d,n})$.

{The catalecticant map behaves well with polynomial vectors.
If $f\in \oplus_{i=1}^r{\mathrm Sym}^{a_i} V $, for any $e\in{\mathbb Z}$ we define the catalecticant map
$C_f\colon {\mathrm Sym}^eV^*\to \oplus_{i=1}^r{\mathrm Sym}^{a_i-e}V$ which is again the contraction by $f$. If $f$ has rank one, this means there exists
$\ell\in V$ and scalars $\lambda^{(i)}$ such that
$f=\left(\lambda^{(1)}\ell^{a_{1}}, \ldots, \lambda^{(r)}\ell^{a_{r}} \right)$  .
It follows that $\mathrm{rk\ }C_f\le 1$, since the image of $C_f$ is generated by
$\left(\lambda^{(1)}\ell^{a_{1}-e}, \ldots, \lambda^{(r)}\ell^{a_{r}-e} \right)$, which is zero
if and only if $a_r<e$. It follows by linearity the basic inequality
$$\mathrm{rk\ }C_f\le\mathrm{rk\ } f.$$ Again the $(k+1)$-minors of $C_f$ vanish on
the variety of polynomial vectors with rank bounded by $k$, which is $\Sec_k(X)$, where $X$ is the projective bundle defined in \S\ref{subsec:projbundle}.

A classical example is the following. Assume $V=\C^3$. London showed in \cite{Lon}(see also \cite{Sco}) that a pencil of ternary cubics
$f=(f_1,f_2)\in {\mathrm Sym}^3V\oplus{\mathrm Sym}^3V$ has border rank $5$
if and only if $\det C_f=0$ where
$C_f\colon  {\mathrm Sym}^2V^*\to V\oplus V$ is represented by a $6\times 6$
matrix (see \cite[Remark 4.2]{CaCh} for a modern reference). Indeed $\det C_f$ is the equation
of $\Sec_5(X)$ where $X$ is the Segre-Veronese variety $\left(\p^1\times\p^2,\o_X(1,3)\right)$. Note that $X$ is $5$-defective according to Definition
\ref{def:secant} and this phenomenon is pretty similar to the case of Clebsch quartics recalled in the introduction.

The following result goes back to Sylvester.
\begin{prop}[Classical Apolarity]
Let $f=\sum_{i=1}^kl_i^d\in Sym^dV$, let $Z=\{l_1,\ldots, l_k\}\subset V$.
Let $C_f\colon {\mathrm Sym}^eV^*\to {\mathrm Sym}^{d-e}V$ be the contraction by $f$. Assume the rank of $C_f$ equals $k$.
Then  $${\mathrm BaseLocus}\ker \left(C_f\right)\supseteq Z.$$
\end{prop}
\begin{proof} Apolarity Lemma (see \cite{RS}) says that $I_Z\subset f^{\perp}$, which reads in degree $e$ as
$H^0(I_Z(e))\subset \ker C_f$. Look at the subspaces in this inclusion as subspaces of $H^0(\p^n,\o(d))$. The assumption on the rank implies that (compare with the proof of \cite[ Prop. 4.3]{OO})
$${\mathrm codim\ }H^0(I_Z(e))\le k={\mathrm rk\ }C_f = {\mathrm codim\ } \ker C_f,$$
hence we have the equality $H^0(I_Z(e))=\ker C_f$. It follows $${\mathrm BaseLocus}\ker \left(C_f\right)
={\mathrm BaseLocus}H^0(I_Z(e))\supseteq Z.$$\end{proof}

Classical Apolarity is a powerful tool to recover $Z$ from $f$,
hence it is a powerful tool to write down a minimal Waring decomposition of $f$.

The following Proposition \ref{prop:nonabelian} is a further generalization and it reduces to classical apolarity when $(X,L)=(\p V,{\o}(d))$
and $E={\o}(e)$ is a line bundle. The vector bundle $E$ may have larger rank and explains the name of  
Nonabelian Apolarity.

We recall that the natural map $H^0(E)\otimes H^0(E^*\otimes L)\to H^0(L)$
induces the linear map $H^0(E)\otimes H^0(L)^*\to H^0(E^*\otimes L)^*$,
then for any $f\in H^0(L)^*$ we have the contraction map
$A_f\colon H^0(E)\to H^0(E^*\otimes L)^*$.

\begin{prop}[Nonabelian Apolarity]\label{prop:nonabelian}\cite[ Prop. 4.3]{OO}
Let $X$ be a variety, $L\in Pic(X)$ a very ample line bundle which gives the embedding
$X\subset \p\left(H^0(X,L)^*\right)=\p W$. Let $E$ be a vector bundle on $X$.
Let $f=\sum_{i=1}^k w_i\in W$ with $z_i=[w_i]\in\p W$, let $Z=\{z_1,\ldots, z_k\}
\subset\p W$.
It is induced $A_f\colon H^0(E)\to H^0(E^*\otimes L)^*$.
Assume that $\mathrm{rk}A_f=k\cdot \mathrm{rk}E$. 
Then
 ${\mathrm {BaseLocus}}\ker \left(A_f\right)\supseteq Z$.
\end{prop}

In all cases we apply the Theorem we will compute separately $\mathrm{rk}A_f$.

Nonabelian Apolarity enhances the power of Classical Apolarity
and may detect a minimal Waring decomposition of a polynomial in some cases when Classical Apolarity fails, see next Proposition \ref{prop:nonabelian_applied}.
Our main examples start with the quotient bundle $Q$ on $\p^n=\p(V)$, it has rank $n$ and it is defined by the Euler exact sequence
$$0\rig{}\o(-1)\rig{}\o\otimes V^*\rig{}Q\rig{}0.$$

Let $L=\o(d)$ and $E=Q(e)$. Any $f\in {\mathrm Sym}^d\C^3$ induces
the contraction map
\begin{equation}\label{eq:contractionq}
A_f\colon H^0(Q(e))\to H^0(Q^*(d-e))^*\simeq H^0(Q(d-e-1))^*.
\end{equation}

The following was the argument used in \cite{OO} to prove cases (ii) and (iii)
of \ref{eq:classiclist}.
\begin{prop}
  \label{prop:nonabelian_applied} Let $X$ be a variety, $L\in Pic(X)$ a
  very ample line bundle and $E$ a vector bundle on
  $X$ with ${\rm rk}E=\dim X$. Let $[f]\in \p(H^0(L)^*)$ be a general point, $k=
  \frac{h^0(X,L)}{\dim X+1}$, and 
$A_f\colon H^0(E)\to H^0(E^*\otimes L)^*$ the {contraction} map.
Assume that $\mathrm{rk}A_f=r\cdot \mathrm{rk}E$, and 
$c_{\rm rkE}(E)=k$. Assume moreover that
for a specific $f$ the base locus of $\ker A_f$ is given by $k$ points.
Then the $k$-secant map 
$$\pi_k:\sec_k(X)\to\p(H^0(L)^*)$$
is birational. The assumptions are verified in the following cases, corresponding to (ii) and (iii) of (\ref{eq:classiclist}).
$$\begin{array}{c|c|c|c}
  (X,L) &H^0(L)& \textrm{rank\ }&E\\
\hline\\
 (\p^2,\o(5))& {\mathrm Sym}^5\C^3& 7&Q_{\p^2}(2)\\
(\p^3,\o(3))&{\mathrm Sym}^3\C^4& 5&Q_{\p^3}^*(2)\\
\end{array}$$
\end{prop}

Specific $f$'s in the statement may be found as random polynomials in \cite{M2}.
In order to prove also cases (iv) and (v) of (\ref{eq:classiclist})
and moreover our Theorem \ref{thm:main334} we need to extend this result
as follows

\begin{thm}
  \label{thm:nonabelian_applied2} Let $X\rig{\pi} Y$ be a projective bundle, $L\in Pic(X)$ a
  very ample line bundle and $F$ a vector bundle on
  $Y$, we denote $E=\pi^*F$. Let $[f]\in \p(H^0(L)^*)$ be a general point, $k=
  \frac{h^0(X,L)}{\dim X+1}$, and 
$A_f\colon H^0(E)\to H^0(E^*\otimes L)^*$ the {contraction} map.
Let $a=\dim\ker A_f$. Assume that $\mathrm{rk}A_f=k\cdot \mathrm{rk}E$ and that
 $(c_{\mathrm{rk }F}F)^{a}=k$. Assume moreover that
for a specific $f$ the base locus of $\ker A_f$ is given by $k$ 
fibers of $\pi$.
Then  the $k$-secant map 
$$\pi_k:\sec_k(X)\to\p(H^0(L)^*)$$
is birational.  The assumptions are verified in the following cases.
{\footnotesize
$$
\begin{array}{l|l|c|c|l}
  (X,L) &H^0(L)& \textrm{rank\ }&F&\dim\ker A_f\\
\hline\\
\left(\p\left(\oplus_{i=1}^r\o_{\p^1}(a_i)\right),\o_X(1)\right)&\oplus_{i=1}^r
{\mathrm Sym}^{a_i}\C^2&
{\tiny k:=\frac{\sum_{i=1}^r a_i+1}{r+1}}
&\o_{\p^1}(k)&1 (\textrm{if\ }k\le a_1+1)\\
\left(\p\left(\o_{\p^2}(2)^4\right),\o_X(1)\right)& ({\mathrm Sym}^2\C^3)^{\oplus 4}&
4&\o_{\p^2}(2)&2\\
\left(\p\left(\o_{\p^2}(2)\oplus\o_{\p^2}(3)\right),\o_X(1)\right)&{\mathrm Sym}^2\C^3\oplus {\mathrm Sym}^3\C^3&4&\o_{\p^2}(2)&2
\\
\left(\p\left(\o_{\p^2}(3)^2\oplus\o_{\p^2}(4)\right),\o_X(1)\right)&\left({\mathrm Sym}^3\C^3\right)^{\oplus{2}}\oplus{\mathrm Sym}^4\C^3&7&Q_{\p^2}(2) &1
\end{array}$$}
\end{thm}

\begin{proof} By Proposition \ref{prop:nonabelian} we have
 $Z\subset Baselocus(\ker A_f)$, where the base locus can be found by the common zero locus of some sections $s_1,\ldots, s_a$
of $E$ which span $\ker A_f$.
Since $E=\pi^*F$ and $H^0(X,E)$ is naturally isomorphic to $H^0(Y,F)$, the zero locus of each section of $E$ corresponds to the pullback
through $\pi$ of the zero locus of the corresponding section of $F$.
By the assumption on the top Chern class of $F$ we expect 
that the base locus of $\ker A_f$ contains $k=\mathrm{length\ }(Z)$ fibers of the projective bundle $X$. The hypothesis guarantees that this expectation
is realized for a specific polynomial vector $f$. By semicontinuity, it is realized for the generic $f$. This determines the forms $l_i$ in (\ref{eq:simwaring}) for a generic polynomial vector $f$. It follows that $f$ is in the linear span of the fibers $\pi^{-1}(l_i)$ where $Z=\{l_1,\ldots, l_a\}$. Fix representatives for the forms $l_i$ for $i=1,\ldots, k$. Now the scalars $\lambda_i^j$
in (\ref{eq:simwaring})
are found by solving a linear system. Our assumptions imply that $X$ is not $k$-defective, otherwise
the base locus of $\ker A_f$ should be positive dimensional. In particular the tangent spaces at points in $Z$,
which are general, are independent by Terracini Lemma. Since each $\pi$-fiber  is contained in the corresponding tangent space,
it follows that the fibers $\pi^{-1}(l_i)$ corresponding to $l_i\in Z$ are independent. It follows that the scalars $\lambda_i^j$
in (\ref{eq:simwaring}) are uniquely determined and we have generic identifiability.  The check that the assumptions are verified in the cases listed has been perfomed
with random polynomials with the aid of Macaulay2 package \cite{M2}.
In all these cases, by the projection formula we have the natural isomorphism $H^0(X,E^*\otimes L)\simeq H^0(Y,F\otimes\pi_*L)$. 
\end{proof}

Note that the first case in the list of Theorem \ref{thm:nonabelian_applied2} corresponds to Ciliberto-Russo Theorem
\ref{thm:binforms0}, in this case $H^0(E)={\mathrm Sym}^k\C^2$ has dimension $k+1$,
$H^0(E^*\otimes L)={\mathrm Sym}^{a_1-k}\C^2\oplus\ldots\oplus {\mathrm Sym}^{a_r-k}$
has dimension $\sum_{i=1}^r(a_i-k+1)= k$ (if $k\le a_1+1)$ and the contraction map $A_f$ has rank $k$,
with one-dimensional kernel. 

The last case in the list of Theorem \ref{thm:nonabelian_applied2} corresponds to Theorem \ref{thm:main334}.
A general vector $f\in ({\mathrm Sym}^3\C^3)^{\oplus 2}\oplus {\mathrm Sym}^4\C^3$
induces the contraction $A_f\colon H^0(Q(2))\to H^0(Q)\oplus H^0(Q)\oplus H^0(Q(1))$
with one-dimensional kernel. Each element in the kernel vanishes on $7$ points which give the seven  Waring summands of $f$.
}

Note also that $\left(\p\left(\o_{\p^2}(2)^4\right),\o_X(1)\right)$
coincides with Segre-Veronese variety $(\p^3\times\p^2,\o(1,2))$

\begin{rem}\label{rem:catalecticant} The assumption  $ a_{1}+1 \geq k $ in \ref{thm:binforms0}
is equivalent to $\frac{1}{r+1}\sum_{i=1}^r(a_i+1)\le a_1+1$
which means that $a_i$ are ``balanced''.
\end{rem}

We conclude this section showing how the existence of a unique decomposition determines the
birational geometry of the varieties parametrizing higher rank
decompositions. 
The following is just a slight generalization of \cite[Theorem 4.4]{MM}
\begin{thm}
  \label{thm:unirat} Let $X\subset\p^N$ be such that the $k$-secant
  map $\pi_k:\sec_k(X)\to \p^N$ is birational. Assume that $X$ is
  unirational then for $p\in \p^N$
  general the variety $\pi_h^{-1}(p)$  is unirational for any $h\geq
  k$, in particular it is irreducible.
\end{thm}
\begin{proof} Let $p\in \p^N$ be a general point, then for $h>k$ we
  have  $\dim
  \pi_h^{-1}(p)=(h+1)\dim X-1-N=(h-k)(\dim X+1)$. Note that, for $q\in
  \p^N$ general, a general
  point in $x\in\pi^{-1}_h(q)$ is uniquely associated to a set of $h$
  points $\{x_1,\ldots,x_h\}\subset X$ and an $h$-tuple
  $(\lambda_1,\ldots,\lambda_h)\in{\C^h}$ with the requirement that
$$q=\sum\lambda_ix_i .$$ Therefore the birationality of $\pi_k$
 allows to associate, to a general point in $q\in \p^N$, its unique
 decomposition in sum of $k$ factors. That is
 $\pi_k^{-1}(q)=(q,[\Lambda_k(q)])$ for a general point $q\in\p^N$.
Via this identification  we may define a map
$$\psi_h:(X\times\p^1)^{h-k}\map\pi^{-1}_h(p) $$
given by
\begin{eqnarray*}(x_1,\lambda_1,\ldots,x_{h-k},\lambda_{h-k})\mapsto
  (p,[\langle x_1,\ldots,x_{h-k},\Lambda_k(p-\lambda_1x_1-\ldots-\lambda_{h-k}x_{h-k})\rangle]).
\end{eqnarray*}
The map $\psi_h$ is clearly generically finite, of degree
${h}\choose{n+1}$, and dominant. This is sufficient to show the claim.
\end{proof}
Theorem \ref{thm:unirat} applies to all  decompositions
that admit a unique form
\begin{cor}
  Let $f=(f_1,\ldots,f_r)$ be a vector of general homogeneous
  forms. If $f$ has a unique  Waring
  decomposition of rank $k$. Then the set of decompotions of rank
  $h>k$  is parametrized by a unirational variety.
\end{cor}

\begin{rem} Let's go back to our starting example (\ref{eq:2quadrics}) and specialize $f_1=\sum_{i=0}^nx_i^2$ to the euclidean quadratic form.
Then any minimal Waring decomposition of $f_1$ consists of $n+1$ orthogonal summands, with respect to the euclidean form.
It follows that the decomposition (\ref{eq:2quadrics}) is equivalent to the diagonalization of $f_2$ with orthogonal summands.
Over the reals, this is possible for any $f_2$ by the Spectral Theorem.

Also Robert's Theorem, see (v) of (\ref{eq:classiclist}), has a similar interpretation. If $f_1=x_0^2+x_1^2+x_2^2$ and $f_2\in{\mathrm Sym}^3\C^3$ is general, the unique
Waring decomposition of the pair $(f_1,f_2)$ consists in chosing four representatives of lines $\{l_1,\ldots, l_4\}$
and scalars $\lambda_1,\ldots, \lambda_4$ such that \begin{equation}\label{eq:robertortho}
\left\{\begin{array}{rcl}f_1&=&\sum_{i=1}^4l_i^2\\
\hspace{0.2cm} \\
f_2&=&\sum_{i=1}^4\lambda_il_i^3\end{array}\right.
\end{equation} Denote by $L$ the $3\times 4$ matrix whose $i$-th column is given by the coefficients of $l_i$.
Then the first condition in (\ref{eq:robertortho}) is equivalent to the equation
\begin{equation}\label{eq:llt}LL^t=I.\end{equation}

This equation generalizes orthonormal basis and the column of $L$ makes a {\it Parseval} frame, according to
\cite{CMS} \S 2.1. So Robert's Theorem states that the general ternary cubic has a unique decomposition
consisting of a Parseval frame.

In general a Parseval frame for a field $F$ is given by $\{l_1,\ldots, l_n\}\subset F^d$ such that
the corresponding $d\times n$ matrix $L$ satisfies the condition $LL^t=I$. This is equivalent to the equation
$\sum_{i=1}^n(\sum_{j=1}^d l_{ji}x_j)^2=\sum_{i=1}^dx_i^2$, so again to a Waring decomposition with $n$ summands
of the euclidean form in $F^d$.
This makes a connection of our paper with \cite{ORS}, which studies frames in the setting of secant varieties and tensor decomposition.
For example equation (7) in \cite{ORS} define a solution to (\ref{eq:llt}) with the additional condition that the four columns have unit norm.
Note that equation (8) in \cite{ORS} define a Waring decomposition of the pair $(f_1, T)$. Unfortunately the additional condition about unitary norm
does not allow to transfer directly the results of \cite{ORS} to our setting, but we believe this connection deserves to be pushed further.

It is interesting to notice that the decompositions of moments $M_2$ and $M_3$ in \cite[\S 3]{AGHKT} is a (simultaneous) Waring decompositions of the quadric $M_2$ and the cubic $M_3$.
\end{rem}

\section{Computational approach}\label{sec:compapp}

In this section we describe how we can face Question 1 and Question 2, introduced in \S\ref{sec:intr}, from the computational analysis point of view.\\ 
\indent With the aid of Bertini \cite{Be}, \cite{BHSW} and Macaulay2 \cite{M2} software systems, we can construct algorithms, based on homotopy continuation techniques and monodromy loops, that, in the spirit of \cite{HOOS}, yield the number of Waring decompositions of a generic polynomial vector $ f = (f_{1}, \ldots, f_{r}) \in {\mathrm Sym}^{a_1} \C^{n+1} \oplus \ldots \oplus {\mathrm Sym}^{a_r} \C^{n+1} $ with high probability. Precisely, given $ n,r,a_{1}, \ldots, a_{r}, k \in \mathbb{N} $ satisfying (\ref{eq:perfect}) and coordinates $ x_{0}, \ldots, x_{n} $, we focus on the polynomial system
\begin{equation}\label{eq:polsys}
\left\{
\begin{array}{l}
f_{1} = \lambda_{1}^{1}\ell_{1}^{a_{1}}+ \ldots +  \lambda_{k}^{1}\ell_{k}^{a_{1}} \\
\quad \, \, \,  \vdots \\
f_{r} = \lambda_{1}^{r}\ell_{1}^{a_{r}}+ \ldots +  \lambda_{k}^{r}\ell_{k}^{a_{r}}\\
\end{array}
\right.
\end{equation}
where $ f_{j} \in {\mathrm Sym}^{a_j} \C^{n+1} $ is a fixed general element, while $ \ell_{i} = x_{0}+ \sum_{h=1}^{n}l_{h}^{i}x_{h} \in \mathbb{P}(\C^{\vee}) $ and $ \lambda_{i}^{j} \in \C $ are unknown. By expanding the expressions on the right hand side of (\ref{eq:polsys}) and by applying the identity principle for polynomials, the $j$-th equation of (\ref{eq:polsys}) splits in $ {{a_{j}+n} \choose {n}} $ conditions. Our aim is to compute the number of solutions of $ F_{(f_{1}, \ldots, f_{r})}([l_{1}^{1}, \ldots, l_{n}^{1}, \lambda_{1}^{1}, \ldots, \lambda_{1}^{r}], \ldots \ldots, [l_{1}^{k}, \ldots, l_{n}^{k}, \lambda_{k}^{1}, \ldots, \lambda_{k}^{r}]) $, the square non linear system of order $ k(r+n)$, arising from the equivalent version of
(\ref{eq:polsys}) in which in each equation all the terms are on one side of the equal sign. In practice, to work with general $ f_{j} $'s, we assign random complex values $ \overline{l}_{h}^{i} $, $ \overline{\lambda}_{i}^{j} $ to $ l_{h}^{i} $, $ \lambda_{i}^{j} $ and, by means of $ F_{(f_{1}, \ldots, f_{r})} $, we compute the corresponding $ \overline{f}_{1}, \ldots, \overline{f}_{r} $, the coefficients of which are so called \emph{start parameters}. In this way, we know a solution $ ([\overline{l}_{1}^{1}, \ldots,\overline{l}_{n}^{1}, \overline{\lambda}_{1}^{1}, \ldots, \overline{\lambda}_{1}^{r}], \ldots \ldots, [\overline{\lambda}_{1}^{k}, \ldots,\overline{l}_{n}^{k}, \overline{\lambda}_{k}^{1}, \ldots, \overline{\lambda}_{k}^{r}]) \in \C^{k(r+n)}$ of $ F_{(\overline{f}_{1}, \ldots, \overline{f}_{r})} $, i.e. a Waring decomposition of $ \overline{f} = (\overline{f}_{1}, \ldots, \overline{f}_{r}) $, which is called a \emph{startpoint}. Then we consider $ F_{1} $ and $ F_{2} $, two square polynomial systems of order $ k(n+r) $ obtained from $ F_{(\overline{f}_{1}, \ldots, \overline{f}_{r})} $ by replacing the constant terms with random complex values. We therefore construct 3 segment homotopies
$$ H_{i} : \C^{k(r+n)} \times [0,1] \to \C^{k(r+n)} $$
for $ i \in \{0,1,2\} $: $H_{0}$ between $ F_{(\overline{f}_{1}, \ldots, \overline{f}_{r})} $ and $ F_{1} $, $ H_{1} $ between $ F_{1} $ and $ F_{2} $, $ H_{2} $ between $ F_{2} $ and $ F_{(\overline{f}_{1}, \ldots, \overline{f}_{r})} $. Through $ H_{0} $, we get a \emph{path} connecting the startpoint to a solution of $ F_{1} $, called \emph{endpoint}, which therefore becomes a startpoint for the second step given by $ H_{1} $, and so on. At the end of this loop, we check if the output is a Waring decomposition of $ \overline{f} $ different from the starting one. If this is not the case, this procedure suggests that the case under investigation is identifiable, otherwise we iterate this technique with these two \emph{startingpoints}, and so on. If at certain point, the number of solutions of $ F_{(\overline{f}_{1}, \ldots, \overline{f}_{r})} $ stabilizes, then, with high probability, we know the number of Waring decompositions of a generic polynomial vector in $ {\mathrm Sym}^{a_1} \C^{n+1} \oplus \ldots \oplus {\mathrm Sym}^{a_r} \C^{n+1} $. \\
We have implemented the homotopy continuation technique both in the software Bertini\cite{Be}, opportunely coordinated with Matlab, and in the software Macaulay2, with the aid of the package \emph{Numerical Algebraic Geometry}\cite{KL}. \\
\indent Before starting with this computational analysis, we need to check that the variety $ \mathbb{P}(\mathcal{O}_{\mathbb{P}^{n}}(a_{1}) \oplus \ldots \oplus \mathcal{O}_{\mathbb{P}^{n}}(a_{r})) $, introduced in \S\ref{sec:secant}, is not $ k $-defective, in which case (\ref{eq:polsys}) has no solutions. In order to do that, by using Macaulay2, we can construct a probabilistic algorithm based on Theorem \ref{th:terracini}, that computes the dimension of the span of the affine tangent spaces to $ \mathbb{P}(\mathcal{O}_{\mathbb{P}^{n}}(a_{1}) \oplus \ldots \oplus \mathcal{O}_{\mathbb{P}^{n}}(a_{r})) $ at $ k $ random points and then we can apply semicontinuity properties. \\
\indent In the following table we summarize the results we are able to obtain combining numerical and theoretical approaches. Our technique is as follows. We first apply the probabilistic algorithm, checking $ k $-defectivity, described above. If this suggests positive $k$-defect $ \delta_{k} $, we do not pursue the computational approach. When $ \delta_{k} $ is zero, we apply homotopy continuation technique. If the number of decompositions (up to order of summands) stabilizes to a number, $ \symbol{35}_{k} $, we indicate it. If homotopy technique does not stabilize to a fixed number, we apply degeneration techniques like in \S\ref{sec:ternary} to get a lower bound. If everything fails, we put a question mark. Bold degrees are the one obtained via theoretical arguments.\\

\begin{longtable}{|c|c|l|c|c|c|} 
\hline \multicolumn{1}{|c|}{$r$} & \multicolumn{1}{|c|}{$n$}& \multicolumn{1}{|l|}{ $(a_{1},\ldots,a_{r})$}& \multicolumn{1}{|c|}{$k$}&  \multicolumn{1}{|c|}{$\delta_k$}& \multicolumn{1}{|c|}{$\symbol{35}_{k}$}\\
\hline
\endhead
\hline
\endfoot
$2$ &$2$ &$(4,5)$ & $ 9 $ &0 &$ 3 $ \\
$2$ &$2$ &$(6,6)$ & $ 14 $ &0 & $ \geq 2 $ \\
$2$ &$2$ &$(6,7)$ & $ 16 $ &0& $ \geq 8 $\\
$2$ &$3$ &$(2,4)$ & $ 9$ &$2$&  \\
$3$ &$2$ &$(2,2,6)$ & $ 8 $ & $4$&  \\
$3$ &$2$ &$(3,3,4)$ & $ 7 $ & 0&${\bf 1} $ \\
$3$ &$2$ & $(3,4,4)$ & $ 8 $ & 0&$ 4 $ \\
$3$ &$2$ & $(5,5,6)$ & $ 14 $ & 0&$ 205$ \\
$3$ &$3$ & $(3,3,3)$ & $ 10 $ & 0&$ 56 $ \\
$4$ &$2$ & $(2,2,4,4)$ & $ 7 $ &  $2$&  \\
$4$ &$2$ & $(2,3,3,3)$ & $ 6 $ & 0&$ 2$ \\
$4$ &$2$ & $(4,\ldots,4)$ & $ 10 $ &0 & $ ?$ \\
$5$ &$2$ & $(5,\ldots,5,6)$ & $ 16 $ &0 &  $ ? $\\
$6$ &$2$ & $(2,\ldots,2,3)$& $ 5 $ & $3$ &  \\
$6$ &$4$ & $(2.\ldots,2)$ & $ 9 $ & 0&$ 45 $ \\
$7$&$3$&$(2,\ldots,2)$&$7$&0&$\bf 8$\\
$8$&$2$&$(3,\ldots,3)$&$8$&$0$&$\bf 9$\\
$8$&$2$&$(2,\ldots,2,6)$&$7$&$ 7$& \\
$11$&$4$&$(2,\ldots,2)$&$11$&$0$& ${\bf 4368}$\\
$13$&$2$&$(4,\ldots,4)$&$13$&$0$& ${\bf 560}$\\
$15$&$2$&$(4,\ldots,4,6)$&$14$& $6$& \\
$17$&$3$&$(3,\ldots,3)$&$17$& $0$& $ {\bf 8436285}$\\
$19$&$2$&$(5,\ldots,5)$&$19$& $0$& ${\bf 177100}$\\
$26$&$2$&$(6,\ldots,6)$&$26$& $0$& ${\bf 254186856}$ \\
\end{longtable}

\section{Identifiability of pairs of ternary forms}\label{sec:ternary}

In this section we aim to study the identifiability of pairs of ternary forms.
In particular we study the special case of two forms of degree $a$ and $a+1$.
 Our main result is the following
\begin{thm}
  \label{th:main_identifi} Let $a$ be an integer then a general pair of ternary forms of degree $a$ and $a+1$ is identifiable if and only if $a=2$. Moreover there are finitely many decompositions if and only if $a=2t$ is even, and for such an $a$ the number of decompositions is at least
$$ \frac{(3t-2)(t-1)}2+1. $$
\end{thm}

The Theorem has two directions on one hand we need to prove that $a=2$ is identifiable, on the other we need to show that for $a>2$ a general pair is never identifiable. The former is a classical result we already recalled in (iii) of (\ref{eq:classiclist}) and
in Theorem \ref{thm:nonabelian_applied2}. For the latter observe that $\dim\sec_k(X)=4k-1$, therefore if either $4k-1<N$ or $4k-1>N$ the general pair is never identifiable. We are left to consider the perfect case $N=4k-1$.
Under this assumption we may assume that $X$ is not $k$-defective, we
will prove that this is always the case in
Remark~\ref{rem:non_defective}, otherwise the non identifiability is immediate. Hence the core of the question is to study generically finite maps
$$\pi_k:\sec_k(X)\to\p^N,$$
with $4k=(a+2)^2$. This yields our last numerical constraint that is
$a=2t$ needs to be even.

The first step is borrowed from, \cite{Me}\cite{Me1}, and it is
slight generalization of \cite[Theorem 2.1]{Me}, see also \cite{CR}.

\begin{thm}\label{th:birational_tangent_proj} Let $X\subset\p^N$ be an irreducible variety of dimension $n$. 
Assume that the natural map $\sigma:\sec_{k}(X)\to\p^N$ is dominant and generically finite of degree $d$. Let $z\in\Sec_{k-1}(X)$ be a general point.  
Consider $\f:\p^N\dasharrow\p^n$ the projection from the embedded tangent
space $\T_z\Sec_{k-1}(X)$. Then
$\f_{|X}:X\dasharrow\p^n$ is dominant and generically finite of degree at most $d$.
\label{th:proj}
\end{thm}
\begin{proof}
  Choose a general point $z$ on a general $(k-1)$-secant linear space spanned by $\langle
  p_1,\ldots,p_{k-1}\rangle$.
Let $f:Y\to \p^N$ be the blow up of $\sec_{k-1}(X)$ with exceptional
divisor $E$, and fiber  $F_z=f^{-1}(z)$. 
Let $y\in F_z$ be a general point. This point uniquely determines a
linear space $\Pi$  of dimension $(k-1)(n+1)$ that
contains $\T_z\sec_{k-1}(X)$. 
Then the projection $\f_{|X}:X\dasharrow \p^n$ is generically finite of degree $d$ if and only if
$(\Pi\setminus \T_z\sec_{k-1}(X))\cap X$ consists
of just $d$   points.

Assume that $\{x_1,\ldots, x_a\}\subset (\Pi\setminus
\T_z\sec_{k-1}(X))\cap X$.
By Terracini Lemma, Theorem \ref{th:terracini},
$$\T_z\sec_{k-1}(X)= \langle \T_{p_1}X,\ldots, \T_{p_{k-1}}X\rangle$$
Consider the linear spaces $\Lambda_i=\langle x_i, p_1,\ldots,p_{k-1}\rangle$. The Trisecant
Lemma, see for instance \cite[Proposition 2.6]{ChCi}, yields
$\Lambda_i\neq\Lambda_j$, for $i\neq j$. Let
$\Lambda_i^Y$,  and $\Pi^Y$ be the strict transforms on $Y$. 
Since $z\in\langle p_1,\ldots,p_{k-1}\rangle$ and $y=\Pi^Y\cap F_z$ then  $\Lambda_i^Y$ contains the point $y\in F_z$. In particular we have
$$\Lambda_i^Y\cap\Lambda_j^Y\neq\emptyset$$ 

Let $\pi_1:\Sec_k(X)\to\p^N$ be the morphism from the abstract secant variety,
and  $\mu:U\to Y$ the induced morphism. That is
$U=\Sec_k(X)\times_{\p^N} Y$. Then 
there exists a commutative diagram
$$\diagram
U\dto_{p}\rto^{\mu}&Y\dto^{f}\\
\Sec_k(X)\rto^{\pi_1}&\p^N\enddiagram$$

Let $\lambda_i$ and $\Lambda^U_i$ be the strict transform of $\Lambda_i$ in $\Sec_k(X)$ and $U$ 
respectively. By Remark \ref{re:vuoto}
 $\lambda_i\cap \lambda_j=\emptyset$, so that
$$\Lambda^U_i\cap \Lambda^U_j=\emptyset.$$

This proves that $\sharp{\mu^{-1}(y)}\geq a$. But $y$ is a general point of a divisor in the normal variety $Y$. Therefore $\deg\mu$, and henceforth $\deg\pi_1$, is at least $a$.
\end{proof}
To apply Theorem~\ref{th:birational_tangent_proj} we need to better understand $X$ and its tangential projections. 
By definition we have
$$X\simeq\p((\o_{\p^2}(-1)\oplus\o_{\p^2})\otimes\pi^*\o_{\p^2}(a+1))$$
then $X\subset\p^N$ can be seen as the embedding of $\p^3$ blown up in one point $q$ embedded by monoids of degree $a+1$ with vertex $p$. That is let $\L=|\I_{q^a}(a+1)|\subset|\o_{\p^3}(a+1)|$, and $Y=Bl_q\p^3$ then
$$X=\f_\L(Y)\subset\p^N.$$

It is now easy, via Terracini Lemma, to realize that the restriction
of the tangential projection $\f_{|X}X\dasharrow\p^3$ is given by the linear system
$$\sH=|\I_{q^a\cup p_1^2\ldots\cup p_{k-1}^2}(a+1)|\subset|\o_{\p^3}(a+1)|.$$
We already assumed  that $X$ is not $k$-defective  that is we work under the condition

\noindent$(\dag)$\hspace{5cm} $\dim\sH=3.$
\begin{rem}
  It is interesting to note that for $a=2$ the map $\f_{|X}$ is the
  standard Cremona transformation of $\p^3$ given by
  $(x_0,\ldots,x_3)\mapsto (1/x_0,\ldots,1/x_3)$.
\end{rem}

Let us work out a preliminary Lemma, that we reproduce by the lack of an adequate reference.
\begin{lem}
  \label{lem:birational_degeneration}  Let $\Delta$ be a complex disk around
  the origin, $X$ a variety and $\o_X(1)$ a base point free line bundle. 
Consider the product $V=X\times \Delta$,
 with the natural projections, $\pi_1$ and $\pi_2$. Let $V_t=X\times\{t\}$ and
 $\o_V(d)=\pi_1^*(\o_{X}(d))$.
Fix a configuration $p_1,\ldots,p_l$ of $l$ points on $V_0$ and let
 $\sigma_i:\Delta\to V$ be sections such that $\sigma_i(0)=p_i$ and
 $\{\sigma_i(t)\}_{i=1,\ldots,l}$ are general points of $V_t$ for
$t\neq 0$. Let $P=\cup_{i=1}^l\sigma_i(\Delta)$, and $P_t=P\cap V_t$.

 Consider the linear system $\sH=|\o_{V}(d)\otimes\I_{P^2}|$ on $V$, with  $\sH_t:=\sH_{|V_t}$. 
Assume that $\dim\sH_0=\dim\sH_t=\dim X$, for $t\in\Delta$. Let
$d(t)$ be the degree of the map induced by $\sH_t$. Then $d(0)\leq d(t)$.
\end{lem}
\begin{proof}
If, for $t\neq 0$,  $\f_{\sH_t}$ is not dominant the claim is clear. Assume that $\f_{\sH_t}$ is dominant for $t\neq 0$. Then $\f_{\sH_t}$ is generically finite and  $\deg\f_{\sH_t}(X)=1$, for $t\neq 0$. 
Let $\mu:Z\times\Delta\to V$ be a resolution of the base locus, $V_{Zt}=\mu^*V_t$, and  $\sH_{Z}=\mu^{-1}_*\sH$ the strict transform linear systems on $Z$. 
Then $V_{Zt}$ is a blow up of $V_t=X$, for $t\neq 0$, and $V_{Z0}=\mu^{-1}_*V_{0}+R$, for some effective, eventually trivial, residual divisor $R$.
By hypothesis $\sH_0$ is the flat limit of $\sH_t$,  for $t\neq 0$. 
Hence flatness forces 
$$d(t)=\sH_Z^{\dim X}\cdot V_{Zt}=\sH_Z^{\dim X}\cdot(\mu^{-1}_*V_{0}+R)\geq \sH_Z^{\dim X}\cdot \mu^{-1}_*V_{0}=d(0).$$
\end{proof}

Lemma~\ref{lem:birational_degeneration} allows us to work on a degenerate configution to study the degree of the map induced by $|\I_{q^a\cup p_1^2\ldots\cup p_{k-1}^2}(a+1)|\subset|\o_{\p^3}(a+1)|\subset|\o_\p^3(a+1)|$.

\begin{lem}
  \label{lem:degeneration_ok} Let $H\subset \p^3\setminus\{q\}$ be a plane, $B:=\{p_1,\ldots,p_b\}\subset H$ a set of $b:=1/2t(t+3)$ general points, and $C:=\{x_1,\ldots,x_c\}\subset\p^3\setminus\{q\cup H\}$ a set of 
$c:=1/2t(t+1)$ general points. Let $a=2t$ and
$$\sH:=|\I_{q^a\cup C^2\cup B^2}(a+1)|\subset|\o_{\p^3}(a+1)|,$$
be the linear system of monoids with vertex $q$ and double points along $B\cup C$, and $\f_\sH$ the associated map. 
Then $\dim \sH=3$ and 
$$\deg\f_{\sH}> \frac{(3t-2)(t-1)}2.$$
\end{lem}
\begin{proof} Note that by construction the lines $\Span{q,p_i}$ and $\Span{q,x_i}$ are contained in the base locus of $\sH$.
  Let us start computing $\dim\sH$. First we prove that there is a unique element in $\sH$ containing the plane $H$. 
  \begin{claim}\label{cl:1}
    $|\sH-H|=0$.
  \end{claim}
  \begin{proof}[Proof of the Claim]
Let $D\in\sH$ be such that $D=H+R$ for a residual divisor in $|\o(a)|$. Then $R$ is a cone with vertex $q$ over a plane curve $\Gamma\subset H$. Moreover $R$ is singular along $C$ and has to contain $B$. This forces $\Gamma$ to contain $B$ and to be singular at $\Span{q,x_1}\cap H$. In other words $\Gamma$ is a plane curve of degree $2t$ with $c= 1/2t(t+1)$ general double points and passing thorugh $b= 1/2t(t+3)$ general points. Note that 
$$\binom{2t+2}2-3c-b=1.$$
It is well known, see for instance \cite{AH}, that the $c$ points impose independent conditions on plane curves of degree $2t$. Clearly the latter $b$ simple points do the same therefore there is a unique plane curve $\Gamma$ satisfying the requirements. This shows that $R$ is unique and in conclusion the claim is proved.
  \end{proof}

We are ready to compute the dimension of $\sH$
\begin{claim}\label{cl:2}
  $\dim\sH=3$
\end{claim}
\begin{proof}[Proof of the Claim] 
The expected dimension of $\sH$ is 3. Then by  Claim~\ref{cl:1} it is enough to show that $\dim\sH_{|H}=2$. To do this observe that
$\sH_{|H}$ is a linear system of plane curves of degree $2t+1$ with $b$ general double points and $c$ simple general points. As in the proof of Claim~\ref{cl:1} we compute the expected dimension
$$\binom{2t+3}2-3b-c=3,$$
and conclude by \cite{AH}.
\end{proof}

Next we want to determine the base locus scheme of $\sH_{|H}$.
Let $\epsilon:S\to H$ be the blow up of $B$ and $\Span{q,x_i}\cap H$, with $\sH_S$ strict transform linear system.
We will first prove the following.

\begin{claim}
 The scheme base locus of  $|\I_{B^2}(2t+1)|\subset|\o_{\p^2}(2t+1)|$ is $B^2$.
\end{claim}
\begin{proof}
Let $\sL_{ij}:=|\I_{B\setminus\{p_i,p_j\}}(t)|\subset|\o_{\p^2}(t)|$, then  
$$\dim \sL_{ij}=\binom{t+2}2-b-2-1=2.$$
By the Trisecant Lemma, see for instance \cite[Proposition 2.6]{ChCi}, we conclude that 
$$\Bs\sL_{ij}=B\setminus\{p_i,p_j\}.$$
Let $\Gamma_i, \Gamma_j\in\sL_{ij}$ be such that $\Gamma_i\ni p_i$ and $\Gamma_j\ni p_j$.  Then by construction we have
$$D_{ij}:=\Gamma_i+\Gamma_j+\Span{p_i,p_j}\in\sH.$$
Let $D_{ijS}$, $\sL_{ijS}$ be the strict transforms on $S$.
 Note that $\Gamma_h$ belongs to a pencil of curves in $\sL_{hk}$ for any $k$. These pencils do not have common base locus outside of $B$ since $\sL_{ijS}$ is base point free and $\dim \sL_{ij}=2$. Therefore the $D_{ijS}$ have no common base locus.
\end{proof}

\begin{claim}
  $\sH_{S}$ is base point free.
\end{claim}
\begin{proof} To prove the Claim it is enough to prove that the simple base points associated to $C$ impose independent conditions. Since $C\subset\p^3$ is general this is again implied by  the Trisecant Lemma.
\end{proof}

 Then we have
$$\deg\f_{\sH_{S}}=\sH_{S}^2=(2t+1)^2-4b-c=\frac{(3t-2)(t-1)}2.$$
To conclude observe that, with the same argument of the claims, we can prove that $\f_{\sH|R}$ is generically finite, therefore  
$$\deg\f_{\sH}>\deg\f_{\sH|H}=\deg\f_{\sH_{S}}=(2t+1)^2-4b-c=\frac{(3t-2)(t-1)}2 $$
\end{proof}

\begin{rem}\label{rem:non_defective}Lemma~\ref{lem:degeneration_ok}
  proves that $\deg\f_\sH$ is finite. Hence as a byproduct we get that condition $(\dag)$ is always satisfied in our range. That is $X$ is not $k$-defective for $a=2t$.
\end{rem}

\begin{proof}[Proof of Theorem~\ref{th:main_identifi}] We already know
  that the number of decomposition is finite only if $a=2t$. By
  Remark~\ref{rem:non_defective} we conclude that the number is finite
  when $a=2t$. Let $d$ be the number of decompositions for a general
  pair. Then by Theorem~\ref{th:birational_tangent_proj} we know that
  $d\geq \deg\f$ where $\f:X\dasharrow \p^3$ is the tangential
  projection. The required bound is obtained combining Lemma~\ref{lem:birational_degeneration} and Lemma~\ref{lem:degeneration_ok}.  
\end{proof}

\end{document}